\documentclass[12pt]{amsart}
\usepackage{amsmath}%
\usepackage{amsfonts}%
\usepackage{amssymb}%
\usepackage{graphicx}
\usepackage{tikz-cd}
\usepackage{enumerate}
\usepackage{rotating}
\usepackage{rotfloat}
\usepackage{caption}
\usepackage[all]{xy}

\newtheorem{teo}{Theorem}[section]
\newtheorem{lemma}[teo]{Lemma}
\newtheorem{cor}[teo]{Corollary}

\newtheorem{defi}[teo]{Definition}
\newtheorem{example}[teo]{Example}
\newtheorem{obs}[teo]{Observation}

\newtheorem{prob}[teo]{Problem}

\floatstyle{plain}
\floatname{diagram}{Diagram}
\newfloat{diagram}{tbp}{lop}[section]

\begin{document}

\title[Colored Tverberg Theorems for non-prime powers]{Colored Tverberg Theorems\\ for non-prime powers}


\author[\hspace{0.5cm} L. V. Mauri]{Leandro V. Mauri}
\address{Departamento de Matem\'atica\\
	Instituto de Ci\^encias Matem\'aticas e de Computa\c c\~ao\\
	S\~ao Paulo University (USP)- C\^ampus de S\~ao Carlos \\
	 S\~ao Carlos, SP, Brazil}
\email{leandro.mauri@usp.br}

\author[R. T. \v{Z}ivaljevi\'{c}]{Rade T. \v{Z}ivaljevi\'{c}}
\address{Mathematical Institute\\
	Serbian Academy of Sciences and Arts (SASA)\\
	Belgrade, Serbia}
\email{rade@turing.mi.sanu.ac.rs }

\author[ D. De Mattos]{Denise de Mattos}
\address{Departamento de Matem\'atica\\
	Instituto de Ci\^encias Matem\'aticas e de Computa\c c\~ao\\
	S\~ao Paulo University (USP)- C\^ampus de S\~ao Carlos \\
	S\~ao Carlos, SP, Brazil}
\email{deniseml@icmc.usp.br}

\author[E. L. dos Santos]{Edivaldo L. dos Santos}
\address{Departamento de Matem\'atica\\
	Universidade Federal de S\^{a}o Carlos\\
	Federal University of S\~{a}o Carlos  (UFSCAR) - C\^ampus de S\~ao Carlos \\
	 S\~ao Carlos, SP, Brazil}
\email{edivaldo@ufscar.br}

\thanks{Leandro V. Mauri was supported by FAPESP of Brazil Grant number 2018/23928-2}

\thanks{R. \v Zivaljevi\' c was supported by the Science Fund of the Republic of Serbia, Grant No.\ 7744592, Integrability and Extremal Problems in Mechanics, Geometry and Combinatorics - MEGIC}

\begin{abstract} We prove a relative of both the \emph{original} and the \emph{optimal (Type B)}  version of the  Colored Tverberg theorem of \v{Z}ivaljevi\'{c} and Vre\'{c}ica
(Theorems \ref{teo:ZV} and \ref{teo:ZV-B}), which modifies these results in two different ways.

(1)  We extend the original theorems beyond the prime powers by showing that the theorem is valid if the number of rainbow faces is $q= p^n-1$.

(2)  The size of some rainbow simplices  may be  smaller than in the original theorems. More precisely $\vert C_i\vert \in \{2q-2, 2q+1\}$  while (for comparison) in the original theorems it is $\vert C_i\vert = 2q-1$.

\smallskip

The proof relies on equivariant index theory and a result of Volovikov [10] about partial coincidences of maps $f : X \rightarrow \mathbb{R}^d$, from a $G$-space into the
Euclidean space.  \end{abstract}


\maketitle

\section{Introduction}
\label{sec:intro-new}

Let $K\subseteq 2^{[m]}$ be a simplicial complex  (with $m$ vertices).
A continuous map  $f : K \rightarrow \mathbb{R}^d$ is called an \emph{almost $r$-embedding}
if $f(\Delta_1)\cap \dots\cap f(\Delta_r) = \emptyset$  for each collection $\{\Delta_i\}_{i=1}^r$ of pairwise disjoint faces of $K$. If an almost $r$-embedding of $K$ in $\mathbb{R}^d$ does not exist we say that $K$ is \emph{not almost $r$-embeddable} in $\mathbb{R}^d$.
The general \emph{Tverberg problem} is to describe interesting classes of simplicial complexes which are or are not almost  $r$-embeddable in $\mathbb{R}^d$. Historically the case of an $N$-dimensional simplex $K=\Delta_N$ was studied first. It is still one of the central research themes, side by side with the case when $K = R_{C_1,C_2,\dots, C_{k+1}}:=  C_1\ast\dots\ast C_{k+1}$ is the join of $0$-dimensional complexes (the Colored Tverberg problem).

\subsection{Almost $r$-embedding for non prime powers}
\label{sec:beyond-prime-powers}

It is known \cite{Fr15} \cite{BFZ} that almost $r$-embeddability (or non-embeddability) of a simplicial complex is critically dependent on the arithmetical properties of $r$. More precisely $r$ is assumed to be a prime power $r = p^n$  in the majority of  results of this type.

\medskip
What if $r$ is not a prime power? For example if $K=\Delta_N$ is an  $N$-dimensional simplex then, as documented in the following  results,
the non-prime power case holds only if we substantially increase the dimension of the simplex.

\bigskip\noindent
$\bullet_1$  \quad $\Delta_N$ is not almost $r$-embeddable in $\mathbb{R}^d$ if  $r = p^\nu$ is a prime power, $d\geq 1$, and $N = (r-1)(d+1)$.

\medskip\noindent
\hfill (I. B\' ar\' any, S.B. Shlosman, A.~Sz\H ucs. 1981; M. \"{O}zaydin 1987; A.Yu. Volovikov 1996; etc.)

\medskip

\bigskip\noindent
$\bullet_2$ \quad $\Delta_{r(d+1)-1}$ is not almost $r$-embeddable in $\mathbb{R}^d$ for all $r\geq 2$ and $d\geq 1$.

\medskip\noindent\hfill (F. Frick and P. Soberon [FS20])

\bigskip\noindent
$\bullet_3$  \quad $\Delta_{(r-1)(d+1)}$ is almost $r$-embeddable in $\mathbb{R}^d$ if  $r$ is not a prime power and $d\geq 2r+1$.

\medskip\noindent
\hfill (\cite{Fr15,BBZ,BFZ,BS17,Sh18,Sk16})

\bigskip\noindent
$\bullet_4$ \quad $\Delta_N$ is almost $r$-embeddable in $\mathbb{R}^d$  if $r$ is not a prime power and $N = (d+1)r - r\left \lceil{\frac{d+2}{r+1}}\right \rceil -2$

\medskip\noindent\hfill (S. Avvakumov, R. Karasev and A. Skopenkov \cite{AKS19})



\bigskip
All these results are  instances of the following general problem:
  Determine integers $a$ and $d$ such that there exists (or there does not exist) an almost $r$-embedding $\Delta_a \rightarrow \mathbb{R}^d$.
  All of them illustrate the fact that the case when $r$ is not prime power is more subtle and currently in the mainstream of research in this area.

\medskip
It appears that the \emph{Colored Tverberg problem} \cite{Z17} 
was somewhat neglected and not directly affected by these developments. In particular, much less is  known about the  almost $r$-non embeddability  of
``rainbow complexes''
$$K = R_{C_1,C_2,\dots, C_{k+1}} := C_1\ast C_2\ast\dots\ast C_{k+1}$$
if $r$ is not a prime power.

\medskip
Our Theorem \ref{teo:first}  is an example of such an  extension where:

\begin{itemize}
    \item[(1)] the number of intersecting rainbow faces is  $q = p^n-1$;
    \item[(2)] $\vert C_1\vert = \vert C_2\vert = \dots = \vert C_m\vert = 2q+1$, $\vert C_{m+1}\vert = \dots = \vert C_{k+1-m}\vert = 2q-2$, under the condition \begin{equation} m \ge (d-k)(p^n-1) = (d-k)q \, .
    \label{eq:cond}
        \end{equation}
\end{itemize}
If $k = d$ the condition (\ref{eq:cond}) disappears and we observe (Corollary \ref{cor:first}) that the result is valid if $m = 0$.
This is a slight improvement over Theorem \ref{teo:ZV}, where $$\vert C_1\vert = \vert C_2\vert = \dots = \vert C_{d+1}\vert = 2r-1 \, .$$
(Note however that neither Theorem \ref{teo:ZV} nor Theorem \ref{teo:ZV-B} is formal
consequence of Theorem \ref{teo:first}.)

Examples \ref{ex:first} and \ref{ex:second}) illustrate some special, low-dimensional cases of Theorem \ref{teo:first} which indicate that this result should be often close to the optimal in the case when the number of rainbow simplices is $p^n-1$.

\section{An overview of topological Tverberg type results}

The following result is known as the topological Tverberg theorem.

 \begin{teo}[{Topological Tverberg theorem, \cite{Bar81}, \cite{Oza87}, \cite{Vol96}}] \label{teo:TTT}
 Let $d \ge 1$, $r \ge 2$, and $N= (r-1) (d+1)$ be integers. If $r$ is a prime power, then for any continuous map $f: \Delta_N \rightarrow \mathbb{R}^{d}$ there are $r$ pairwise disjoint faces $\sigma_1, \ldots , \sigma_r$ of $\Delta_N$ such that $f(\sigma_1) \cap \cdots \cap f(\sigma_r) \neq \emptyset.$

\end{teo}

\vspace{0.2cm}
Interesting problems and (conjectured) extensions and relatives  of the Topological Tverberg theorem have emerged over the years. In particular, motivated by questions from discrete and computational geometry, Bárány and Larman \cite{Bar92} formulated in 1992 the \emph{colored Tverberg problem}.

\vspace{0.2cm}

\begin{defi}[{Coloring}] Let $N \ge 1$ be an integer and let $V(\Delta_N)$ be the set of vertices of the simplex $\Delta_N$. A \textit{coloring} of vertices of $V(\Delta_N)$ by $l$ colors is a partition $(C_1, \ldots , C_l)$ of $V(\Delta_N)$, that is $V(\Delta_N) = C_1 \cup \cdots \cup C_l$, with $C_i \cap C_j = \emptyset$, for $1 \le i < j \le l$. The elements of the partition $(C_1, \ldots , C_l)$ are called \textit{color classes}.
\end{defi}

\begin{defi}[{Rainbow face}] Let $(C_1, \ldots , C_l)$ be the coloring of $V(\Delta_N)$ by $l$ colors. A face $\sigma$ of the  simplex $\Delta_N$ is a \textit{rainbow face} if $|\sigma \cap C_i| \le 1$, for all $1 \le i \le l$.

\vspace{0.2cm}

\end{defi}

\begin{prob}[{Bárány-Larman colored Tverberg problem}] Let $d \ge 1$ and $r \ge 2$ be integers. Determine the smallest number $n= n(d,r)$ such that for every map $f: \Delta_{n-1} \rightarrow \mathbb{R}^{d}$, and every coloring $(C_1, \ldots , C_{d+1})$ of the vertex set $V(\Delta_{n-1})$ of the simplex $\Delta_{n-1}$ by $d+1$ colors, with each color of size at least $r$, there exist $r$ pairwise disjoint rainbow faces $\sigma_1, \ldots , \sigma_r$ of $\Delta_{n-1}$ satisfying: \begin{eqnarray} \hspace{0.2cm} f(\sigma_1) \cap \cdots \cap f(\sigma_r) \neq \emptyset. \nonumber\end{eqnarray}

\end{prob}

A modified colored Tverberg problem was presented by \v{Z}ivaljevi\'{c} and Vre\'{c}ica in \cite{Ziv92}.

\begin{prob}[{\v{Z}ivaljevi\'{c}-Vre\'{c}ica colored Tverberg problem}] Let $d \ge 1$ and $r \ge 2$ be integers. Determine the smallest number $t=t(d,r)$ such that for every affine (or continuous) map $f: \Delta \rightarrow \mathbb{R}^{d}$, and every coloring $(C_1, \ldots , C_{d+1})$ of the the vertex set $V(\Delta)$ by $d+1$ colors, with each color of size at least $t$, there exist $r$ pairwise disjoint rainbow faces $\sigma_1, \ldots , \sigma_r$ of $\Delta_{n-1}$ satisfying: \begin{eqnarray}  f(\sigma_1) \cap \cdots \cap f(\sigma_r) \neq \emptyset.\nonumber\end{eqnarray}
\end{prob}

For $r \ge 2$ a prime power, \v{Z}ivaljevi\'{c} and Vre\'{c}ica proved that $t(d,r) \le 2r-1$. This result is known as the \emph{(original) Colored Tverberg theorem of \v{Z}ivaljevi\'{c} and Vre\'{c}ica}.

\vspace{0.2cm}

\begin{teo}[{Colored Tverberg theorem of \v{Z}ivaljevi\'{c} and Vre\'{c}ica} \cite{Ziv92}]
\label{teo:ZV}
Let $d \ge 1$ be an integer, and let $r =p^{n}\ge 2$ be a prime power. For every continuous map $f: \Delta \rightarrow \mathbb{R}^{d}$, and every coloring $(C_1, \ldots , C_{d+1})$ of the the vertex set $V(\Delta)$ by $d+1$ colors, with each color of size at least $2r-1$, there exist $r$ pairwise disjoint rainbow faces $\sigma_1, \ldots , \sigma_r$ of $\Delta$ satisfying: \begin{eqnarray} f(\sigma_1) \cap \cdots \cap f(\sigma_r) \neq \emptyset.\nonumber\end{eqnarray}

\end{teo}

\vspace{0.2cm}
The following result is known as \emph{Optimal (Type B) Colored Tverberg theorem of \v{Z}ivaljevi\'{c} and Vre\'{c}ica}, see [13] and [11].

\vspace{0.2cm}

\begin{teo} [{Optimal (Type B) Colored Tverberg theorem of \v{Z}ivaljevi\'{c} and Vre\'{c}ica}]
\label{teo:ZV-B}
Assume that $r=p^\nu$ is a prime power,  $d\geq 1$, and let $k$ be an integer such that $\frac{r-1}{r}d \leq k<  d$.  Then the complex
\[ R_{C_0, C_1,\dots, C_k}:= C_0\ast \dots\ast C_k\] is not almost $r$-embeddable in $\mathbb{R}^d$ if $\vert C_i\vert \geq 2r-1$ for all $i$.
\end{teo}

\section{Topological preliminaries}

In this section we collect  central definitions and results needed for the proof of our main theorem.

\subsection{Configuration spaces}

\emph{Deleted joins} and \emph{deleted products} are the standard \emph{configuration spaces} used, in the framework of the \emph{configuration space/test map scheme}  \cite{mat08, Z17, Bla17},  in applications of topological methods to problems of combinatorics and discrete and computational geometry.

\vspace{0.2cm}

\begin{defi}[{Deleted join}] Let $K$ be a simplicial complex, let $n \ge 2$, $k \ge 2$ be integers, and let $[n]=\{1, \ldots, n\}$. The \textit{$n$-fold $k$-wise deleted join} of the simplicial complex $K$ is the simplicial complex:
\begin{eqnarray}
K_{\Delta(k)}^{\ast n} =\bigcup \left\{  \sigma_1 \ast \cdots \ast \sigma_n \subset K^{\ast n} \mid  (\forall I \subset [n])\, \vert I\vert \ge k \Rightarrow \bigcap_{i \in I} \sigma_i = \emptyset  \right\}\nonumber
\end{eqnarray}
 where $\sigma_1, \ldots , \sigma_n$ are faces of $K$, including the empty face. The symmetric group $\mathcal{G}_{n}= \mbox{Sym} (n)$ acts on $K_{\Delta(k)}^{\ast n}$ by: \begin{eqnarray} \pi \cdot ( \lambda_1 x_1 + \cdots + \lambda_n x_n ) = \lambda_{\pi^{-1}(1)} x_{\pi^{-1}(1)} + \cdots + \lambda_{\pi^{-1}(n)} x_{\pi^{-1}(n)},  \nonumber \end{eqnarray}  for $\pi \in \mathcal{G}_{n}$ and $\lambda_1 x_1 + \cdots + \lambda_n x_n \in K_{\Delta(k)}^{\ast n}$.

\end{defi}

\vspace{0.2cm}

\begin{defi}[{Deleted product}] Let $K$ be a simplicial complex, let $n \ge 2$, $k \ge 2$ be integers, and let $[n]=\{1, \ldots, n\}$. The \textit{$n$-fold $k$-wise deleted product} of the simplicial complex $K$ is the cell complex:
\begin{eqnarray} 
K_{\Delta(k)}^{\times n} = \bigcup\left\{\sigma_1 \times \cdots \times \sigma_n \subset K^{\times n} \mid ( \forall I \subset [n])\, \vert I\vert \ge k \Rightarrow \bigcap_{i \in I} \sigma_i = \emptyset  \right\}\nonumber
\end{eqnarray} 
where $\sigma_1, \ldots , \sigma_n$ are non-empty faces of $K$. The symmetric group $\mathcal{G}_{n}= \mbox{Sym} (n)$ acts on $K_{\Delta(k)}^{\times n}$ by: \begin{eqnarray} \pi \cdot ( x_1, \ldots , x_n) = ( x_{\pi^{-1}(1)} , \ldots , x_{\pi^{-1}(n)}),  \nonumber \end{eqnarray}  for $\pi \in \mathcal{G}_{n}$ and $( x_1 , \ldots, x_n ) \in K_{\Delta(k)}^{\times n}$.

\end{defi}

\vspace{0.2cm}

\begin{defi}[{Chessboard complex}] The \textit{$m \times n$ chessboard complex} $\Delta_{m,n}$ is the simplicial complex whose vertex set is $[m] \times [n]$, and the simplexes of $\Delta_{m,n}$ are the subsets $\{(i_0,j_0), \ldots, (i_k,j_k) \} \subset [m] \times [n]$, where $i_s \neq i_{s'}$ ($1 \le s < s'\le k$), and $j_t \neq j_{t'}$ ($1 \le t < t'\le k$).

\end{defi}

\vspace{0.2cm}

\begin{defi}[{Rainbow subcomplex}] Let $\Delta$ be a simplex with a coloring $(C_1 , \ldots, C_{d+1})$ by $(d+1)$ colors.  We define the \textit{rainbow subcomplex} $R_{(C_1, \ldots, C_{d+1})} \subset \Delta$ as follows: \begin{eqnarray} R_{(C_1, \ldots, C_{d+1})} \cong C_1 \ast \cdots \ast C_{d+1},\nonumber \end{eqnarray} where $C_i$ is a discrete set of points, for every $i \in [d+1]$.

\end{defi}

\vspace{0.2cm}

\subsection{Volovikov index}

The following fundamental result of cohomology theory is used in the definition the Volovikov index of a $G$-space $X$.

\vspace{0.2cm}

\begin{teo}[{The cohomology Leray-Serre Spectral sequence (Theorem $5.2$ \cite{McC01})}]
Let $R$ be a commutative ring with the unity. Given a fibration $F \hookrightarrow E \stackrel{p}{\rightarrow} B$, where $B$ is a path-wise connected space, there is a first quadrant spectral sequence of algebras $\{ E_r^{\ast, \ast}, d_r \}$, with: \begin{eqnarray} E_{2}^{p,q} \cong H^{p} ( B ; \mathcal{H}^{q} (F;R)), \nonumber \end{eqnarray} the cohomology of $B$, with local coefficients in the cohomology of $F$, the fiber of $p$, and converging to $H^{\ast}(E;R)$ as an algebra. Furthermore, this spectral sequence is natural with the respect to fiber-preserving maps of fibrations.

\end{teo}

\vspace{0.2cm}

We continue with the definition of the Volovikov index \cite{Vol00}. It is defined as a function on $G$-spaces (where $G$ is a compact Lie group) whose values are either positive integers or $\infty$. For our application it is sufficient to assume that $G$ is a $p$-torus  $G= (\mathbb{Z}_p)^{n}$, where $p$ a prime number.

\vspace{0.1cm}

\begin{defi}[{Volovikov index}] Let $G$  be a compact Lie group and let $X$ be a Hausdorff paracompact  $G$-space. The definition of the \textit{Volovikov index of $X$}, denoted by $i(X)$, uses the spectral sequence of the bundle $p_X: X_G \rightarrow BG$, with fibre $X$ (the Borel construction), given in Theorem $2.5$. This spectral sequence converges to the equivariant cohomology $H^{\ast} (X_G; \mathbb{Z}_p)$. Let $\Lambda^{\ast}$ be the equivariant cohomology algebra of a point $H^{\ast} (\mbox{pt}_G; \mathbb{Z}_p)= H^{\ast} (BG; \mathbb{Z}_p)$. Suppose that $X$ is path connected. Then $E_{2}^{\ast,0} = \Lambda^{\ast}$. Assume that $E_{2}^{\ast,0}= \cdots = E_{s}^{\ast,0} \neq E_{s+1}^{\ast,0}$. Then, by definition, $i(X)=s$. If $E_{2}^{\ast,0}= \cdots = E_{\infty}^{\ast,0}$ then, by definition, $i(X)= \infty$. Let $i'(X)$ be the least number $r$ such that the kernel of the natural homomorphism $\Lambda^{\ast} \rightarrow E_{r+1}^{\ast, 0}$ contains an element which is not a zero divisor in $\Lambda^{\ast}$.

\end{defi}

\vspace{0.2cm}

The following theorem describes some of the most important properties of the Volovikov index.

\vspace{0.2cm}

\begin{teo} {\rm (\cite{Vol00})} \label{teo:3.7}

$(1)$ If there exists an equivariant map of $G$-spaces $X \rightarrow Y$, then $i(X) \le i(Y)$ and $i'(X) \le i'(Y)$.

\vspace{0.1cm}

$(2)$ If $X$ is a compact or finite-dimensional cohomological sphere (over the the field $\mathbb{Z}_p$), i.e.,$H^{\ast}(X)= H^{\ast}(S^{n})$, and if $G$ acts with no fixed points on $X$, then $i(X)=i'(X)=n+1$.

\vspace{0.1cm}

$(3)$ If $\tilde{H}^{j}(X; \mathbb{Z}_p)=0$, for all $j<n$, then $i(X) \ge n+1$.

\vspace{0.1cm}

$(4)$ If $X= A \cup B$, where $A$ and $B$ are closed (or open) $G$-invariant subespaces, $i(X) \le i'(A) + i(B)$. In particular, $i (X\ast Y) \le i'(X)+ i(Y).$
\end{teo}

\subsection{Connectedness}
\label{sec:connected}

Here we review the definition and some basic properties of the \emph{connectedness} of topological spaces, including a key result which relates the connectedness to the Volovikov index.

\vspace{0.1cm}

\begin{defi}{\rm (\cite[Definition 4.3.2]{mat08})} Let $n \ge -1$ be an integer. A topological space $X$ is \textit{$n$-connected} if any continuous map $f: S^{k} \rightarrow X$, where $-1 \le k \le n$, can be continuously extended to a continuous map $g: B^{k+1} \rightarrow X$, that is $g|_{\partial B^{k+1} = S^k} = f$ (here $B^{k+1}$ denotes a $(k+1)$-dimensional closed ball whose boundary is the sphere $S^{k}$). A topological space is $(-1)$-connected if it is non-empty. If the space $X$ is $n$-connected , but not $(n+1)$-connected, we write $\mbox{conn} (X)=n$.

\end{defi}

\vspace{0.1cm}

\begin{teo}{\rm (\cite[p. 332]{Bla17})}\label{teo:3.9}
 Let $X$ and $Y$ be topological spaces. Then 
 \begin{eqnarray}\mbox{conn} (X \ast Y) \ge  \mbox{conn} (X) + \mbox{conn} (Y) +2 \, . \nonumber \end{eqnarray} 
 \end{teo}

\vspace{0.1cm}

\begin{teo}{\rm (\cite[Theorem 4.4.1]{mat08})} Let $X$ be a nonempty topological space and let $k \ge 1$. Then $X$ is $k$-connected if and only if it is simply connected (i.e., the fundamental group $\pi_1(X)$ is trivial) and $\tilde{H}_i(X)=0$, for all $i=0,1, \ldots,k$. \end{teo}

\vspace{0.1cm}

\begin{teo}\label{teo:3.11} Let $X$ topological space. Then, $i (X ) \ge  \mbox{conn} (X) +2.$
\end{teo}

\vspace{0.1cm}

\textit{Proof.} It is a consequence of Theorem $2.10$ and Theorem $2.7$ $(3)$.

\hspace{14.3cm} $\square$

\vspace{0.1cm}

\begin{teo}{\rm (\cite{blvz})} \label{teo:3.12}
Let $m, n \ge 1$ be integers. Then: 
\begin{eqnarray} \mbox{conn} (\Delta_{m,n}) = \mbox{min} \left\{ m,n, \left \lfloor{\frac{m+n+1}{3}} \right \rfloor \right\}-2.   \nonumber 
\end{eqnarray}

\end{teo}

\vspace{0.2cm}

\section{Colored Tverberg theorem with $p^{n}-1$ faces}
\label{sec:main}

\vspace{0.2cm}

In this section we prove the main result of the paper (Theorem \ref{teo:first}). First, we state and prove two lemmas that are needed for the proof.

\vspace{0.1cm}

\begin{defi}{\rm (\cite{Vol09})} Let $X$ be a $G$-space, where $G$ is a finite group, and let $f: X \rightarrow Y$ be a continuous map. For $2 \le y \le |G|$, we set \begin{eqnarray} A(f,y)= \{ x \in X  \mid f(g_1 x) = \cdots = f(g_y x)  \mbox{ for some distinct } g_i \in G \}.\nonumber \end{eqnarray}

\end{defi}

\vspace{0.2cm}

\begin{lemma}

Let $r= p^{n} \ge 2$ be a prime power and let $d \ge 1$, $1 \le k \le d$, $ 2 \le q \le r$ be integers. For every continuous map $f: \Delta \rightarrow \mathbb{R}^{d}$ and every coloring $(C_1, \ldots , C_{k+1})$ of the vertex set $V(\Delta)$ by $(k+1)$ colors, define the continuous map as follows: \begin{eqnarray} h : (R_{(C_1, \ldots, C_{k+1})})_{\Delta(2)}^{\times p^{n}} \rightarrow \mathbb{R}^{d}\nonumber \\ \hspace{1.5cm} (x_1, \ldots ,x_{p^{n}}) \mapsto f(x_1).\nonumber \end{eqnarray}

If $A(h,q) \neq \emptyset$ then there exists $q$ pairwise disjoint rainbow faces $\sigma_1, \ldots , \sigma_q$ of $\Delta$ such that \begin{eqnarray} f(\sigma_1) \cap \cdots \cap f(\sigma_q) \neq \emptyset \, .\nonumber \end{eqnarray}

\end{lemma}

\vspace{0.4cm}

\textit{Proof.}  Choose $ (x_1, \ldots, x_{p^{n}}) \in A(h,q) \neq \emptyset.$

\vspace{0.2cm}

Then there exist distincts elements $g_1,\ldots, g_q \in (\mathbb{Z}_p)^{n}$ such that \begin{eqnarray} h(g_1(x_1,\ldots , x_{p^{n}})) = \cdots = h(g_q(x_1,\ldots , x_{p^{n}}))\, .\nonumber \end{eqnarray}

\vspace{0.2cm}

Therefore there exist $q$ elements $x_{i_1}, \ldots, x_{i_q} \in  \{x_{1}, \ldots, x_{p^{n}}\}$ such that $f(x_{i_1})= \cdots = f(x_{i_q})$, where $x_{i_1} \in \sigma_{i_1}, \ldots , x_{i_q} \in \sigma_{i_q}$ ($\sigma_{i_{m}}$ is a support of $x_{i_m}$, for every $m \in [q]$).

\vspace{0.2cm}

By the definition of the configuration space $(R_{(C_1, \ldots, C_{k+1})})_{\Delta(2)}^{\times p^{n}}$, there exist $q$ pairwise disjoint,  non-empty rainbow  faces $\sigma_{i_1}, \ldots ,\sigma_{i_q}$  such that \begin{eqnarray} f(\sigma_{i_1}) \cap \cdots \cap f(\sigma_{i_q}) \neq \emptyset\, .\nonumber \end{eqnarray}



\begin{lemma}
\label{lema:index-3-3}
Let $d \ge 1$, $1 \le k \le d$, $0 \le m \le k+1$ be integers and let $r=p^{n} \ge 2$ be a prime power. Let $(C_1, \ldots , C_{k+1})$ be a coloring of the vertex set $V(\Delta)$ by $(k+1)$ colors, where we have $|C_i| \ge 2r-1$, $\forall i =1, \ldots, m$, $|C_i| \ge 2r-4$, $\forall i=m+1, \ldots, k+1$ and $m \ge (d-k)(r-1)$. Then: \begin{eqnarray}i((R_{(C_1, \ldots, C_{k+1})})_{\Delta(2)}^{\times p^{n}}) \ge d (p^{n}-1). \nonumber \end{eqnarray}
\end{lemma}
\begin{proof}
Note that 
\begin{eqnarray} 
(R_{(C_1, \ldots , C_{k+1})})_{\Delta(2)}^{\ast p^{n}} = A \cup B \mbox{,}\nonumber 
\end{eqnarray} 
where 
\begin{eqnarray}
 A= \left\{ \lambda_1 x_1 + \cdots +\lambda_{p^{n}} x_{p^{n}}\in (R_{(C_1, \ldots , C_{k+1})})_{\Delta(2)}^{\ast p^{n}} \mid (\exists i \in [ p^{n}])\, \lambda_i \neq \nonumber \frac{1}{p^{n}}  \right \}
\end{eqnarray}
\begin{eqnarray}  
B= \left\{ \lambda_1 x_1 + \cdots +\lambda_{p^{n}}x_{p^{n}} \in (R_{(C_1, \ldots , C_{k+1})})_{\Delta(2)}^{\ast p^{n}} \mid \lambda_1 = \cdots = \lambda_{p^{n}} =\frac{1}{p^{n}} \right \}.\nonumber
\end{eqnarray}

It is not difficult to see that $B$ is isomorphic to $(R_{(C_1, \ldots , C_{k+1})})_{\Delta(2)}^{\times p^{n}} $.
It follows from Theorem \ref{teo:3.7} (4) that   
\begin{eqnarray} 
i((R_{(C_1, \ldots , C_{k+1})})_{\Delta(2)}^{\ast p^{n}}) \le i'(A) +i(B) = i'(A) + i((R_{(C_1, \ldots , C_{k+1})})_{\Delta(2)}^{\times p^{n}})\, .  \nonumber
\end{eqnarray}
We want to estimate the indices $i((R_{(C_1, \ldots , C_{k+1})})_{\Delta(2)}^{\ast p^{n}})$ and $i'(A)$.
In light of the isomorphism 
\begin{eqnarray} (R_{(C_1, \ldots, C_{k+1})})_{\Delta(2)}^{\ast r} \cong \Delta_{|C_1|,r} \ast \cdots \ast \Delta_{|C_{m}|,r} \ast \Delta_{|C_{m+1}|,r} \ast \cdots \ast \Delta_{|C_{k+1}|,r}. \nonumber 
\end{eqnarray}
 we obtain, as a consequence of Theorems \ref{teo:3.9} and \ref{teo:3.12}, 
 \begin{eqnarray} 
 \mbox{conn} ((R_{(C_1, \ldots, C_{k+1})})_{\Delta(2)}^{\ast r}) \ge m (r-2)+ (k+1-m)(r-3)+2k. \nonumber 
 \end{eqnarray}
It follows from Theorem \ref{teo:3.11}  that 
\begin{equation} \begin{split}
i((R_{(C_1, \ldots, C_{k+1})})_{\Delta(2)}^{\ast r})) &\ge [m (r-2)+ (k+1-m)(r-3)+2k]+2\\ &= (k+1)(r-1)+m \, . \nonumber
\end{split}
\end{equation}
Since $m \ge (d-k)(r-1)$, we have  
\begin{eqnarray} 
i((R_{(C_1, \ldots, C_{k+1})})_{\Delta(2)}^{\ast r}) \ge  (d+1)(r-1) \, .  \nonumber 
\end{eqnarray}
In order to find a bound for $i'(A)$  let us consider the following $(\mathbb{Z}_{p})^{n}$-equivariant map \begin{eqnarray} \phi : A \longrightarrow \mathbb{R}^{p^{n}} \setminus \Delta( \mathbb{R}^{p^{n}}).  \nonumber \\ \lambda_1 x_1 +\cdots + \lambda_{p^{n}} x_{p^{n}} \longmapsto (\lambda_1, \ldots ,\lambda_{p^{n}})\nonumber \end{eqnarray} and \begin{eqnarray} \Pi: \mathbb{R}^{p^{n}} \setminus \Delta( \mathbb{R}^{p^{n}}) \rightarrow (\Delta(\mathbb{R}^{p^{n}}))^{\perp} \setminus \{0\} \rightarrow S( (\Delta(\mathbb{R}^{p^{n}}))^{\perp}) \mbox{,}\nonumber \end{eqnarray} where $\Pi$ is a composition of the projection and deformation retraction.

Here $\Delta(\mathbb{R}^{p^{n}})= \{ (x_1, \ldots, x_{p^{n}}) \in \mathbb{R}^{p^{n}} \mid x_1= \cdots = x_{p^{n}} \}$ is the diagonal subspace of $\mathbb{R}^{p^{n}}$, while $S(V)$  the unit sphere in the real vector space $V$.

\medskip 

It follows that the composition \begin{eqnarray}  \Pi \circ \phi : A \longrightarrow S( (\Delta(\mathbb{R}^{p^{n}}))^{\perp}) \cong S^{p^{n}-2} \nonumber \end{eqnarray} is a $(\mathbb{Z}_p)^{n}$-equivariant map. By Theorem \ref{teo:3.7}, (1) and (2), we conclude that \begin{eqnarray} i'(A) \le i'(S^{p^{n}-2}) = p^{n}-1 \nonumber \end{eqnarray}
and as an immediate consequence, $i ((R_{(C_1, \ldots , C_{k+1})})_{\Delta(2)}^{\times p^{n}}) \ge d (p^{n}-1)$. 
\end{proof}

\begin{teo}{\rm (\cite[Theorem 4]{Vol09})} \label{teo:Vol-4}
Let $X$ be a connected $G$-space, where $G=(\mathbb{Z}_p)^{n}$ is a $p$-torus and $2 \le y \le p^n$, $y \neq 3$. Assume the  inequality $i(X) \ge (m-1) (p^n -1) +y$. Then $A(f,y) \neq \emptyset$ for any continuous map $f: X \rightarrow \mathbb{R}^{m}$.

\end{teo}

\vspace{0.1cm}

\noindent{\textbf{Remark.}} Theorem \ref{teo:Vol-4} is also true for $y=3$ and $r=3,4,5$.

\vspace{0.2cm}

\begin{teo}\label{teo:first}
Let $d \ge 1$, $1 \le k \le d$, $0 \le m \le k+1$ be integers, and let $r= p^{n} \ge 2$ be a prime power. For every continuous map $f: \Delta \rightarrow \mathbb{R}^{d}$, and every coloring $(C_1, \ldots , C_{k+1})$ of the vertex set $V(\Delta)$ by $(k+1)$ colors, such that $|C_i| \ge 2r-1$, $\forall i=1, \cdots , m$, $|C_i| \ge 2r-4$, $\forall i=m+1, \cdots , k+1 $ and $m \ge (d-k)(r-1)$, there exist $q=r-1=p^{n}-1$ pairwise disjoint rainbow faces $\sigma_1, \ldots , \sigma_q$ such that: \begin{eqnarray} f(\sigma_1) \cap \cdots \cap f(\sigma_q) \neq \emptyset.\nonumber \end{eqnarray}
\end{teo}

\begin{proof}   It follows from Lemma $3.2$ that if $A(h,q)$ is non-empty  then there exist $q$ pairwise disjoint, rainbow, non-empty faces $\sigma_1, \ldots , \sigma_q$ such that: \begin{eqnarray} f(\sigma_1) \cap \cdots \cap f(\sigma_q) \neq \emptyset\, .\nonumber \end{eqnarray}
Therefore it remains to be shown that $A(h,q) \neq \emptyset$.

On the other hand this is an immediate consequence of Theorem $3.4$, applied to the $G$-space  $X= (R_{(C_1, \ldots, C_{k+1})})_{\Delta(2)}^{\times p^{n}}$ and the map  $h:(R_{(C_1, \ldots, C_{k+1})})_{\Delta(2)}^{\times p^{n}} \rightarrow \mathbb{R}^{d}$ (as in  Lemma $3.2$), where  $y=q$. Indeed, $(R_{(C_1, \ldots, C_{k+1})})_{\Delta(2)}^{\times p^{n}}$ is connected and $i((R_{(C_1, \ldots, C_{k+1})})_{\Delta(2)}^{\times p^{n}}) \ge d (p^{n}-1)= (m-1) (p^{n}-1)+y $ (by  Lemma $3.3$).
This observation completes the proof of the theorem.
\end{proof}

\begin{cor}\label{cor:first}
Let $d \ge 1$ be an integer, and let $r= p^{n} \ge 2$ be a prime power. For every continuous map $f: \Delta \rightarrow \mathbb{R}^{d}$, and every coloring $(C_1, \ldots , C_{d+1})$ of the vertex set $V(\Delta)$ by $(d+1)$ colors, with each color of size at least $2r-4$, there exist $q=r-1=p^{n}-1$ pairwise disjoint rainbow faces $\sigma_1, \ldots , \sigma_q$ such that: \begin{eqnarray} f(\sigma_1) \cap \cdots \cap f(\sigma_q) \neq \emptyset.\nonumber \end{eqnarray}
\end{cor}

\begin{proof}
Apply Theorem $4.5$ to the case $k=d$ and $m=0$.
\end{proof}

\begin{obs} Note that if $m > (d-k)(r-1)$ then  \begin{eqnarray} i((R_{(C_1, \ldots, C_{k+1})})_{\Delta(2)}^{\ast r}) \ge (d+1)(r-1)+1, \nonumber \end{eqnarray} and there exist $r$ pairwise disjoint rainbow faces $\sigma_1, \cdots , \sigma_r$ such that $f(\sigma_1) \cap \cdots \cap f(\sigma_r) \neq \emptyset$.
This means that the interesting case of the Theorem \ref{teo:first} is when $m = (d-k)(r-1)$.
\end{obs}

\begin{example}\label{ex:first} Let $r=7$, $d=8$, $k=7$ and $m=6$. Then we have $k+1=8$ colors $C_1, \ldots C_8$ where $|C_i| \ge 2r-1 =13$, for $i=1, \ldots,6$ and $|C_7|, |C_8| \ge 2r-4 =10$. Note that the condition $m \ge (d-k)(r-1)$ follows (more especifically we have an equality). Then, there exist $q=r-1=6$ pairwise disjoint rainbow faces $\sigma_1, \sigma_2, \sigma_3, \sigma_4, \sigma_5$ and $\sigma_6$ such that $f(\sigma_1) \cap f(\sigma_2) \cap f(\sigma_3) \cap f(\sigma_4) \cap f(\sigma_5) \cap f(\sigma_6)  \neq \emptyset$
\end{example}
The following example illustrates the  Corollary $4.6$.
\begin{example}\label{ex:second}
Let $d=2$, $r=7$ and let $\mathcal{C} = (C_1, C_2, C_3)$ be a coloring of the vertex set $V(\Delta)$, where $|C_1|=|C_2|=|C_3| =2r-4=10$. Let $f: \Delta \rightarrow \mathbb{R}^{2}$ be a continuous map. Then by Corollary \ref{cor:first}, there exist $6$ pairwise disjoint rainbow faces $\sigma_i (i=1,\dots, 6)$ such that \[f(\sigma_1) \cap f(\sigma_2) \cap f(\sigma_3) \cap f(\sigma_4) \cap f(\sigma_5) \cap f(\sigma_6)  \neq \emptyset \, . \]

\vspace{0.2cm}

\vspace{0.2cm}

\end{example}


\end{document}